\documentclass[a4paper,11pt]{article}
\usepackage{amsfonts,amsmath,amssymb,amsthm, mathrsfs}
\usepackage[latin9]{inputenc}
\usepackage{geometry}
\usepackage{hyperref}
\usepackage{fixmath}
\usepackage{paralist}
\usepackage{graphicx}
\usepackage{verbatim}
\hypersetup{%
pdftitle={},
pdfsubject={},
pdfauthor={},
pdfkeywords={},
plainpages=false,pdfborder={0 0 0}}

\newtheorem{COUNT}{}[]

\theoremstyle{plain}
\newtheorem{theorem}[COUNT]{Theorem}
\newtheorem{proposition}[COUNT]{Proposition}
\newtheorem{lemma}[COUNT]{Lemma}
\newtheorem*{lemma*}{Lemma}

\theoremstyle{remark}

\def\AAa{\mathcal {A}}
\def\BB{\mathcal {B}}

\def\FF{\mathcal {F}}

\def\HH{\mathcal {H}}

\def\MM{\mathcal {M}}
\def\NN{\mathcal {N}}

\def\QQ{\mathcal {Q}}

\def\SSs{\mathcal {S}}

\def\ZZ{\mathcal {Z}}
\def\E{\mathbb {E}}
\def\N{\mathbb {N}}
\renewcommand{\P}{\mathbb {P}} 
\def\R{\mathbb {R}}

\def\Z{\mathbb {Z}}

\def\eref#1{(\ref{#1})}

\DeclareMathOperator{\supp}{supp}
\DeclareMathOperator{\diam}{diam}

\def\leb{\mathscr {L}}

\DeclareMathOperator{\e}{\mathbf{e}}

\def\support{[\chi]}
\def\suppW{\mathcal{B}}

\title{Lyapunov Exponents of Brownian Motion: Decay Rates for Scaled Poissonian Potentials and Bounds}
\author{Johannes Rue\ss \\ Universit\"at T\"ubingen}
\date{}

\begin{document}
\maketitle

{
\renewcommand{\thefootnote}{}
\renewcommand{\thefootnote}{}
}
\begin{center}
\begin{minipage}{13 cm}
\textsc{Abstract:} We investigate Lyapunov exponents of Brownian motion in a nonnegative Poissonian potential $V$. The Lyapunov exponent depends on the potential $V$ and our interest lies in the decay rate of the Lyapunov exponent if the potential $V$ tends to zero. In our model the random potential $V$ is generated by locating at each point of a Poisson point process with intensity $\nu$ a bounded compactly supported nonnegative function $W$.
We show that for sequences of potentials $V_n$ for which $\nu_n \|W_n\|_1 \sim D/n$ for some constant $D > 0$ ($n \to \infty$), the decay rates to zero of the quenched and annealed Lyapunov exponents coincide and equal $c n^{-1/2}$ where the constant $c$ is computed explicitly. Further we are able to estimate the quenched Lyapunov exponent norm from above by the corresponding norm for the averaged potential.
\end{minipage}
\end{center}

\section*{Introduction and Results}

We consider Brownian motion in $\R^d$, $d \in \N$. Let $Z$ be the canonical process on the space $C(\R_{\geq 0},\R^d)$ and $P_x$ be the Wiener measure starting from $x \in \R^d$. By $E_x$ we denote the expectation operator belonging to $P_x$, for $P_0$ and $E_0$ we simply write $P$, $E$.
In the model we consider, $Z$ is moving in a random environment $V$ which is formed by obstacles located at points of a Poisson point process $\chi$ on $\R^d$ with constant intensity $\nu >0$, independent of $Z$. Let $(\Omega,\FF,\P)$ be the probability space on which $\chi$ is defined and let $\E$ be the expectation operator belonging to $\P$. The obstacles are shaped by a measurable bounded compactly supported function $W: \R^d \to \R_{\geq0}$ which, in order to avoid trivialities, we assume not to be almost everywhere equal to zero. The potential $V:\R^d \times \Omega \to [0,\infty]$ is defined as
\[
V(x,\omega):=V(x):=  \sum_{p \in \support}W(x-p)
\]
where $x \in \R^d$, $\omega \in \Omega$ and $\support$ is the support of the random measure $\chi$. $V$ will be called the Poissonian potential generated by $W$ and $\nu$.

We define the Green function for Brownian motion in the environment $\eta + V$, where $\eta\geq 0$ is a constant, see e.g.\ \cite[(2.2.3)]{Sznitman98}: Let $p(t,x,y)$, $t >0$, $x,y \in \R^d$ be the transition probabilities of Brownian motion. Define for $\omega \in \Omega$, $t >0$, $x,y \in \R^d$,
\begin{align*}
	r_\eta(t,x,y,\omega):= p(t,x,y)E_{x,y}^t[\exp\{-\int_0^t\eta + V(Z_s, \omega)ds\}].
\end{align*}
Here $E_{x,y}^t$ is the Brownian bridge from $x$ to $y$ in time $t$. The Green function is defined as $g_\eta: \R^d \times \R^d \times \Omega \to [0,\infty]$,
\begin{align*}
g_\eta(x,y,\omega):=\int_0^\infty r_\eta(s,x,y, \omega)ds
\end{align*}
and can be interpreted as density for the expected occupation times measure for Brownian motion starting in $x$ and being killed at rate $\eta + V$.

Introduce the hitting time $H(y):= \inf\{s \geq 0: Z_s \in \overline{B(y,1)}\}$ where $B(y,1)$ is the open ball of radius $1$ centered at $y$ and $\overline{B(y,1)}$ its closure. Set
\begin{align*}
&e(y,\eta + V):= E[\exp\{- \int_0^{H(y)} \eta + V (Z_s)ds\} ], \quad  a(y,\eta + V):= - \ln e(y,\eta + V).
\end{align*}
For $\omega$ fixed $e(y, \eta+V)$ is the probability that Brownian motion being killed at rate $\eta + V$ survives until it reaches $\overline{B(y,1)}$. The quantity $e(y,\eta + V)$ can also be interpreted as the equilibrium potential of $\overline{B(y,1)}$ relative to $-\frac 1 2 \Delta + \eta + V$ (see \cite[Proposition 2.3.8]{Sznitman98}).

The asymptotic behavior of the Green function $g_\eta(0,y,\omega)$ for $\|y\|_2 \to \infty$ has been studied in \cite{Sznitman94}. Sznitman shows exponential decay and also gives an alternative characterization for the limiting object in terms of the quantity $a$. The following Theorem is a version of \cite[Theorem 5.2.5]{Sznitman98}, which omits the fact that the convergence also holds uniformly for all directions.
\begin{theorem}[{\cite[Theorem 5.2.5]{Sznitman98}}]\label{theo:lya_quenched}
There is a norm $\alpha_{\eta+V}: \R^d \to \R_{\geq 0}$, nonrandom, such that $\P$-a.s.\ for any $y \in \R^d$,
\begin{align*}
\lim_{n \to \infty} - \frac{1}{n} \ln g_\eta(0,ny,\omega)= \lim_{n \to \infty} \frac{1}{n} a(ny, \eta + V) = \alpha_{\eta +V}(y)
\end{align*}
and the limits also hold in $L^1(\P)$.
\end{theorem}
The limiting object $\alpha_{\eta + V}$ is called the quenched Lyapunov exponent for Brownian motion in Poissonian potential. In this article we will consider Lyapunov exponents for various potentials and it is therefore convenient to use the notation '$\alpha_{\eta+V}$' in order to indicate, which potential the Lyapunov exponent belongs to, although $\alpha_{\eta+V}$ is a nonrandom function.

We are also going to study annealed Lyapunov exponents: Sznitman shows in \cite{Sznitman95} exponential decay of the $\P$-averaged Green function $\E g_\eta(0,y,\omega)$ for $\|y\|_2 \to \infty$. Here again, since we don't need more details, we give a version of \cite[Theorem 5.3.4]{Sznitman98} which only considers convergence on fixed directions, not mentioning that this convergence holds uniformly on all directions:
\begin{theorem}[{\cite[Theorem 5.3.4]{Sznitman98}}]\label{theo:lya_annealed}
There exists a norm $\beta_{\eta + V}$ on $\R^d$, nonrandom, such that for $y \in \R^d$,
\begin{align*}
\lim_{n \to \infty} -\frac{1}{n} \ln \E g_\eta(0,ny,\omega)  = \lim_{n \to \infty} - \frac{1}{n} \ln \E e(ny, \eta + V) = \beta_{\eta + V}(y).
\end{align*}
\end{theorem}
Analogous results hold in the discrete setting of simple symmetric random walks on $\Z^d$ in random potentials. Zerner has constructed Lyapunov exponents in the quenched case in \cite{Zerner98}. The annealed Lyapunov exponent is considered by Flury in \cite{Flury07}. A recent contribution is given in \cite{Mourrat11}. For further results, models and quantities related to Brownian motion moving in a Poissonian potential we refer to \cite[Chapter 5 and 7]{Sznitman98}.

The aim of this article is to study the behavior of annealed and quenched Lyapunov exponents when the potential converges to zero. In the case of random walks in random potentials this question has been investigated by Wang in \cite{Wang02}. However, the results there have been improved by Kosygina, Mountford and Zerner, see \cite{Kosygina10}. In the present article we will establish analogous results to those of \cite{Kosygina10}. 

The behavior of Brownian motion in scaled Poissonian potentials $\varphi V$ with scaling function $\varphi: \R \to \R_{\geq 0}$ has been studied for various scaling functions $\varphi$. (See e.g.\ \cite[Chapter 7]{Sznitman98}.) Our results complement results of W\"uthrich in \cite{Wuethrich99} where increasing potentials are considered: W\"uthrich shows for continuous $W$ and for the scaling $\varphi(n)=n \in \N$ that $n^{-1/2}\alpha_{n(\eta + V)}(y)$ converges to the time-constant in the corresponding continuous first-passage percolation model as $n \to \infty$, uniformly in all directions $y \in \R^d$, $\|y\|_2=1$.
 
Potentials converging to zero have been studied for time dependent scalings for example by Merkl and W\"uthrich, see \cite{Merkl01} and the references therein. For small potentials, their result shows that Brownian motion essentially only feels the averaged potential, see the remarks on p.192 and equation (0.14) in \cite{Merkl01}. Our results are in the same spirit.

The present article also covers the case that the obstacles become rarefied. For the case of hard obstacles, i.e.\ $W = \infty$ on a compact set and zero else, and for densities decaying with time $t$ see e.g.\ \cite{Berg05} or \cite{Sznitman90} and the references therein. The exact calculation of $\lim_{t \to \infty} -\frac{(\ln t)^{2/d}}{t} \ln E[\exp\{-\int_0^t V(Z_s)ds\}]$ in \cite[(0.2)]{Sznitman93} in particular illustrates the behavior of this limit under time independent scalings of the function $W$ as well as of the density $\nu$. For time independent scalings of the density of hard obstacles see \cite[(1.2)]{Donsker75}.

In the following let $W_n$, $n \in \N$ be bounded measurable compactly supported functions from $\R^d \to \R_{\geq 0}$ not almost everywhere equal to zero. For $n \in \N$ let $\nu_n > 0$, $\eta_n \geq 0$, and let $V_n$ be the Poissonian potential generated by $W_n$ and $\nu_n$. We want $V_n$ to converge to zero in a suitable way, in particular we want to cover the case when a potential $V$ is multiplied by a sequence of real numbers converging to zero, i.e.\ $V_n = \gamma_n V$ with $\gamma_n \to 0$ for $n \to \infty$. It will turn out that $L^1(\R^d)$ convergence of the functions $W_n$ to zero will be the notion of convergence which allows us to determine quantitative decay rates of the Lyapunov exponents.

We will state explicit decay rates for the quenched and the annealed Lyapunov exponents, moreover, these rates will coincide for the quenched and the annealed case. Under this perspective the difference between the annealed and the quenched picture vanishes if $W_n$ is small enough. As far as we know, it is an open problem whether in high dimensions for small $W$ the quenched and the annealed Lyapunov exponents coincide (see \cite[p. 326]{Sznitman98}), for already solved aspects of this problem in the discrete setting we refer to the references and comments given in \cite{Kosygina10} and in \cite{Mourrat11}.

In order to establish decay rates we will not need any regularity assumptions on the obstacles $W_n$. The only restriction on the obstacles needed is boundedness of the $L^\infty(\R^d)$ norms and of the supports:
\begin{theorem}\label{theo:convergence}
Assume $\sup_{n \in \N}\|n W_n\|_\infty <\infty, \quad \sup_{n \in \N} \diam \supp(W_n)< \infty$. Assume there exists a constant $D \geq 0$ such that
\begin{align*}
\lim_{n \to \infty} n(\eta_n + \nu_n \|W_n\|_1) = D.
\end{align*}
Then for any $y \in \R^d$,
\begin{align*}
\lim_{n \to \infty} \sqrt{n} \alpha_{\eta_n + V_n}(y) 
= \lim_{n \to \infty} \sqrt{n} \beta_{\eta_n + V_n}(y) = \sqrt{2D}\|y\|_2.
\end{align*}
Moreover, the convergences are uniform for $y$ in any compact subset of $\R^d$.
\end{theorem}

The proof of this theorem is divided into two main steps. In order to establish a lower bound on $\liminf_{n \to \infty}\sqrt n \beta_{\eta_n + V_n}(y)$ we will use techniques developed by Kosygina, Mountford and Zerner in \cite{Kosygina10}. In order to establish an upper bound on $\limsup_{n \to \infty} $ $\sqrt n \alpha_{\eta_n + V_n}(y)$ we will estimate the quenched Lyapunov exponent. Sznitman's proof of Theorem \ref{theo:lya_quenched} gives a bound on the quenched Lyapunov exponent, see \cite[Proposition 1.2]{Sznitman94}: For $y \in \R^d$,
\begin{align*}
\alpha_{\eta + V}(y) \leq \sqrt{2(\eta + \lambda_d + \|W\|_\infty \nu \omega_d(a+2)^d)}\|y\|_2
\end{align*}
where $\lambda_d$ stands for the principal Dirichlet eigenvalue of $-\frac{1}{2} \Delta$ in $B(0,1)$, $\omega_d$ is the volume of $B(0,1)$, and $a>0$ is such that $W$ is supported in $\overline{B(0,a)}$.
The bound, we will derive, refines this result and in a heuristic sense shows that Brownian motion prefers moving in a hilly random environment to moving in an averaged environment:
\begin{theorem}\label{prop:alphapois_leq_alphaconst}
For $y \in \R^d$,
\begin{align*}
\alpha_{\eta + V}(y) \leq \alpha_{\eta + \E V}(y) = \sqrt{2 (\eta + \nu \|W\|_1)} \|y\|_2.
\end{align*}
\end{theorem}
This bound corresponds to the bound given in the discrete setting by Zerner in \cite[Proposition 4]{Zerner98} applied to the special case that '$V$ is more variable than $\E V$'. Note that this inequality may also be read as $\alpha_{\eta + V}(y)\leq \sqrt{2 ( \eta + \E[V(0)])} \|y\|_2$. The equality in Theorem \ref{prop:alphapois_leq_alphaconst} is a well known calculation of the Lyapunov exponent for constant potentials.

The estimate of Sznitman as well as the estimate derived in this article can also be interpreted as versions of the Combes-Thomas estimate for our concrete model which do take into account the shape of the potential via $\nu \|W\|_\infty$ and $\nu \|W\|_1$ respectively, see e.g.\ \cite[Chapter 2.4]{Stollmann01} or \cite{Germinet03} for an actual account to this subject. For localization results concerning the present model we refer to \cite{Germinet05} and the references therein.

In order to prove Theorem \ref{prop:alphapois_leq_alphaconst}, in a first step we find a suitable discretization of the function $W$ which enables us to apply Jensen's inequality in the finite dimensional setting. A more direct proof gets available if in addition $W$ is assumed to be continuous. In this case one can use directly the concavity of the functional $F \mapsto a(y,\eta +F)$ on nonnegative continuous functions on $\R^d$ in order to apply a general version of Jensen's inequality established by Perlman (see \cite{Perlman74}) to the Pettis integrable random function $V$. This leads to Theorem \ref{prop:alphapois_leq_alphaconst} without any discretization. However, in order to derive the result in whole generality we are going to use the discretization technique.

\section*{Proof of the Upper Bound}

The upper bound on $\limsup_{n \to \infty} \sqrt n \alpha_{\eta_n+V_n}$ in Theorem \ref{theo:convergence} is a direct consequence of Theorem \ref{prop:alphapois_leq_alphaconst}. Hence we start by proving Theorem \ref{prop:alphapois_leq_alphaconst}.

Let $\leb$ denote the Lebesgue measure on $\R^d$ or $\R$ depending on the context. The following lemma examines continuity properties of $a$:
\begin{lemma}\label{lem:Lambda_continuous}
Let $y \in \R^d$, let $f_n$, $n \in \N$, $f$, $g$ be nonnegative, locally integrable functions on $\R^d$ with $f_n \leq g$ and $f_n \to f$ $\leb$-a.e.\ as $n \to \infty$. Then, as $n \to \infty$,
\begin{align*}
E[ \exp\{ -\int_0^{H(y)}f_n(Z_s)ds\}1_{H(y)< \infty}] \to E[ \exp\{ -\int_0^{H(y)}f(Z_s)ds\}1_{H(y)< \infty}].
\end{align*}
\end{lemma}

\begin{proof}
We will show that $P$-a.s.\ whenever $H(y)< \infty$,
\begin{align}\label{eq:lem:lambda_continuous}
\int_0^{H(y)} f_n(Z_s)ds \to \int_0^{H(y)} f(Z_s)ds  \text{ as } n \to \infty.
\end{align}
Dominated convergence then gives the desired result.

Let $\NN$ be a set with $\leb (\NN)=0$ and $f_n \to f$ on $\NN^c$. We examine the expectation of the occupation times measure of $\NN$ and apply Fubini's theorem
\begin{align*}
E[\int_0^{H(y)} 1_\NN(Z_s)ds]
& \leq E[\int_0^{\infty} 1_\NN(Z_s)ds]
= \int_0^{\infty} E[ 1_\NN(Z_s)] ds \\
& = \int_0^{\infty} \int_\NN p(s,0,z) dz \, ds
= 0.
\end{align*}
Thus $P$-a.s.\ the set $\{0 \leq s \leq H(y): Z_s \in \NN\}$ has zero $\leb$-measure, which assures convergence of the functions $s \mapsto f_n(Z_s)$ to $s \mapsto f(Z_s)$ $\leb$-a.e.. Using $f_n \leq g$ the dominated convergence theorem implies that $P$-a.s.\ whenever $H(y)< \infty$ the convergence in \eref{eq:lem:lambda_continuous} holds.
\end{proof}

This continuity property of $a$ will be crucial for the approximation techniques we are going to use in the proof of 
the following lemma.

\begin{lemma}\label{lem:apois_leq_aconst}
For $y \in \R^d$,
\begin{align*}
\E a(y,\eta + V) \leq a ( y,\eta +  \E V(0)).
\end{align*}
\end{lemma}

\begin{proof}
First we are going to 'discretize' the potential $V$ by approximating $W$ with a suitable sequence of functions which are mostly like step functions. 

Since the support of $W$ is bounded, Lusin's theorem assures the existence of a sequence $(K^{(m)})_{m \in \N}$ of compact sets $K^{(m)} \subset \supp W$ such that $W \vert_{K^{(m)}}$ is continuous and $\leb(\supp W \setminus K^{(m)}) \leq 1/m$. A finite union $\bigcup_{i \in I} K_i$, $I \subset \N$, $|I|$ finite, of compact sets is compact, and continuity of $W$ on a finite number of closed sets $K_i$ implies continuity of $W$ on the union $\bigcup_{i \in I} K_i$ of these sets. Thus without restriction we assume the sequence $(K^{(m)})_{m \in \N}$ to be increasing, i.e.\ $K^{(m')} \subset K^{(m)}$ if $m' <m$.

Set $W^{(m)}:=  W1_{K^{(m)}}$ and let $V^{(m)}$ be the Poissonian potential generated by $W^{(m)}$ and $\nu$. For $x,y \in \R^d$ with $x_i \leq y_i$, $1 \leq i \leq d$, we define the cube $[x,y)$ as $\prod_{i=1}^d [x_i,y_i)$. For $k \in \N$ let $\lfloor x \rfloor_k:= (\lfloor x_i \rfloor_k)_{1\leq i \leq d}:=(\lfloor x_i 2^k \rfloor 2^{-k})_{1\leq i \leq d}$, i.e.\ $x$ rounded to the grid $2^{-k}\Z^d$. We denote the cube with side length $2^{-k}$ and $\lfloor x \rfloor_k$ in the bottom left corner by $Q_k(x):= [\lfloor x \rfloor_k, \lfloor x \rfloor_k + 2^{-k}\e )$, where $\e = (1,1,\ldots, 1) \in \R^d$. Define for $m$, $k \in \N$ the discretization of $W$ as the function $W_k^{(m)}: \R^d\times \R^d \to \R_{\geq 0}$
\begin{align*}
W_k^{(m)}(x,z):= \sup_{\ t \in Q_k(z)} W^{(m)}(x-t),\quad 
V_k^{(m)}(x) := \sum_{p \in [\chi ]} W_k^{(m)}(x,p).
\end{align*}

\begin{figure}[!tb]
\begin{center}
\includegraphics{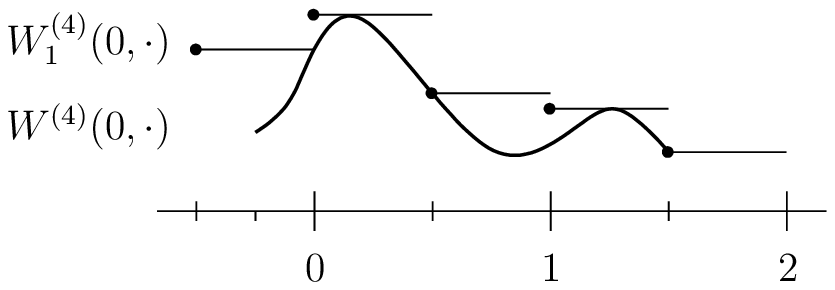}
\end{center}
\caption{Here $W(x) = (f1_{(-1/2,3/2]})(- x)$ where $f:\R \to [0,\infty)$ is continuous. A possible approximation of $W$ is $W^{(m)}(x):= (f1_{[-1/2+1/m,3/2]})(-x)$, $m \in \N$. In that case $\AAa=\{-1/2\}$.}
\end{figure}

We are going to show that there is a sequence $(k(m))_m$ such that for all $x$, $z \in \R^d$ with $x-z \notin \AAa := \supp W \setminus \bigcup_m K^{(m)}$,
\begin{align}\label{W_km_converges_uniformly}
W_{k(m)}^{(m)}(x,z) \to W(x - z) \ \text{as} \ m \to \infty.
\end{align}

Set $k(1):=1$. Assume $k(m-1)$ is already defined. Uniform continuity of $W\vert_{K^{(m)}}$ implies that there is $k(m)>k(m-1)$ such that for $s$, $t \in K^{(m)}$ with $|t-s| \leq 2^{-k(m)}$ one has $|W^{(m)}(t)-W(s)| \leq 1/m$. Hence for all $x$, $z \in \R^d$ with $x-z \in K^{(m)}$
\begin{align}\label{eq:unifcontrol}
&|W_{k(m)}^{(m)}(x,z)-W(x-z)|
= |\sup_{t \in Q_{k(m)}(z)} W^{(m)}(x-t) - W(x-z)| \leq \frac 1 m,
\end{align}
since $|x-t-x+z| \leq 2^{-k(m)}$. In order to verify \eref{W_km_converges_uniformly} one has to consider two cases: If $x-z \in \bigcup_{m}K^{(m)}$ there is $m_0$ such that $x - z \in K^{(m)}$ for $m \geq m_0$. Hence \eref{eq:unifcontrol} gives convergence of $W^{(m)}_{k(m)}(x,z)$ to $W(x-z)$ as $m \to \infty$. If $x-z \notin \supp W$ closedness of $\supp W$ assures the existence of $k_0$ such that $x-t \notin \supp W $ for all $t \in Q_{k_0}(z)$, hence for all $m$ and for all $k \geq k_0$, $W^{(m)}_{k}(x,z) = 0$ and convergence also holds in this case.
Set $W_m:=W^{(m)}_{k(m)}$, $V_m:= V^{(m)}_{k(m)}$, $k = k(m)$.

The approximation of $W$ by $(W_m)_m$ assures convergence of $V_m$ to $V$: In fact, let $x \in \R^d$ and introduce the events $\MM_1(x):= \{ x \notin [\chi]+\AAa\}$ and $\MM_2:=\{ \chi \text{ is locally finite}\}$. Then $\P[(\MM_2)^c]=0$, and $\P[(\MM_1(x))^c] =  \P[\chi(x- \AAa)>0]=0$ since $\leb(\AAa) = 0$. By \eref{W_km_converges_uniformly} on $\MM_1(x) \cap \MM_2$ we have $V_m(x)$ to $V(x)$, hence
\begin{align}\label{eq:conv_forall_x}
\text{for} \ \text{all} \ x \in \R^d \ \P\text{-a.s.} & \quad V_m(x) \to V(x) \ \text{for} \ m \to \infty.
\end{align}
On the other hand local finiteness of $\chi$ implies $\P$-a.s.\ $\leb( [\chi] + \AAa) = 0$. Hence
\begin{align}
\label{eq:conv_forleball_x}
\P \text{-a.s.} \ \text{for} \ \leb \text{-a.e.} \ x \in \R^d & \quad V_m(x) \to V(x) \ \text{for} \ m \to \infty.
\end{align}

That the quantity $a$ is 'compatible' with these convergences will be shown in the following: Choose $x = 0$, let $\BB:= \{ z + w: z,w \in \R^d, z \in \supp W, \ w \ \text{or} \ -w \in Q_1(0)\}$, and define $V_\infty(\cdot) := \sum_{p \in [\chi ]} \|W\|_\infty 1_{\BB}(  \cdot - p)$. Since $V_m (0) \leq V_\infty(0)$ by dominated convergence \eref{eq:conv_forall_x} gives $\E V_m(0) \to \E V(0)$. Interpreting $a(y,r)$, $r\in \R_{\geq 0}$, as the Laplace transform of the hitting time $H(y)$, continuity of Laplace transforms implies
\begin{align}\label{eq:a(EV_m)toa(EV)}
a(y,\eta + \E V_m(0)) \to a(y,\eta + \E V(0)) \ \text{for} \ m \to \infty.
\end{align}
Applying Lemma \ref{lem:Lambda_continuous} and \eref{eq:conv_forleball_x} gives $\P$-a.s.\ convergence of
$
a(y, \eta + V_m)$ to $a(y, \eta + V) \ \text{for} \ m \to \infty.
$
The fact that $V_m \leq V_\infty$, $\P$-integrability of $a(y,\eta + V_\infty)$ (use \cite[(5.2.33)]{Sznitman98}) and dominated convergence imply
\begin{align}\label{eq:Ea(V_m)toEa(V)}
\E a(y, \eta + V_m) \to \E a(y, \eta + V) \ \text{for} \ m \to \infty.
\end{align}

For $N > 0$ introduce the event
$\NN := \{ \|Z_s\|_\infty \leq N ~ \ \text{for all} \ s \leq H(y)\}$ and define $a(y,\eta + V,\NN) := - \ln E[ \exp\{- \int_0^{H(y)} \eta + V(Z_s) ds\},  \NN]$. The discrete properties of the potential $V_m$ enable us to deduce for $N > 0$
\begin{align}\label{eq:aVN_meq_aEV}
\E a(y,\eta + V_m,\NN) \leq a(y,\eta + \E V_m(0),\NN).
\end{align}
Indeed, since for all $x \in \R^d$ the function $W_m(x,\cdot)$ is constant on the sets $ Q_{k}(z)$, $z \in 2^{-k}\Z^d$, we get
\begin{align*}
a(y,\eta + V_m,\NN)
&=- \ln E[ \exp\{- \int_0^{H(y)} \eta + \sum_{p \in [\chi ]} W_m(Z_s,p) ds\},  \NN]\\
&=- \ln E[ \exp\{- \eta H(y) - \sum_{z \in \ZZ } \# Q_{k}(z) \int_0^{H(y)} W_m (Z_s,z) ds\},  \NN],
\end{align*}
where $\# Q_{k}(z)$ denotes the number of points of $\chi$ in $Q_{k}(z)$, and $\ZZ :=  \{ z \in 2^{-k}\Z^d : \|z\|_\infty \leq N+\sup\{\|t\|_\infty: t \in \supp W \}+  2^{-k}\}$. Indeed, on $\NN$ for all $z \notin \ZZ$ the coefficient $w_z:=\int_0^{H(y)} W_m(Z_s,z)ds=0$.
Define the mapping $\Lambda: \R_{\geq 0}^\ZZ \to \R$,
\begin{align*}
\Lambda(u) := -\ln E[ \exp\{ - \eta H(y) - \sum_{z \in \ZZ} w_z u_z \},\NN].
\end{align*}
$\Lambda$ is concave: For $u \in \R^\ZZ$ and $\gamma \geq 0$ set $f_\gamma(u):= \exp\{ - \gamma(\eta H(y) + w \cdot u) \}1_\NN$. Consider $\gamma \in (0,1)$, $u, v \in \R^\ZZ$, then H\"older's inequality implies
\begin{align*}
E[ f_\gamma(u) f_{1-\gamma}(v)]
& \leq E[ f_\gamma(u)^{1/\gamma}]^\gamma E[f_{1-\gamma}(v)^{1/(1-\gamma)}]^{1-\gamma}
 = E[ f_1(u)]^\gamma E[f_1(v)]^{1-\gamma}.
\end{align*}
This together with the fact that $\Lambda(\gamma u +(1-\gamma) v)=-\ln E[f_\gamma(u)f_{1-\gamma}(v)]$ shows concavity of $\Lambda$.
Moreover $w_z$ does not depend on the underlying Poisson point process, therefore by Jensen's inequality
\begin{align*}
\E a(y,\eta + V_m,\NN)
& =  \E \Lambda((\#Q_k(z))_{z \in \ZZ})
 \leq  \Lambda(( \E \# Q_k(z))_{z \in \ZZ})\\
& = - \ln E[\exp\{- \eta H(y) - \int_0^{H(y)} \E[ \sum_{z \in \ZZ } \# Q_k(z)  W_m(Z_s,z) ]ds  \},  \NN]\\
&= - \ln E[ \exp\{-\int_0^{H(y)} \eta + \E V_m(Z_s) ds \},  \NN]
\end{align*}
which proves \eref{eq:aVN_meq_aEV}.

We get
$
\E a(y,\eta + V_m)
\leq \E a ( y,\eta + V_m, \NN)
\leq a ( y,\eta +  \E V_m(0), \NN)
$.
Taking the limit $N \to \infty$ the monotone convergence theorem implies
$\E a(y,\eta + V_m) \leq a ( y,\eta +  \E V_m(0))$.
\eref{eq:a(EV_m)toa(EV)} and \eref{eq:Ea(V_m)toEa(V)} now show the statement.
\end{proof}

The potential, we received by the discretization in the previous proof, appeared in a slightly different way in literature: In fact, we also could have discretized $W^{(m)}$ by
\begin{align*}
\tilde W^{(m)}_k(x,z) := \sup_{s \in Q_k(x), \ t \in Q_k(z)} W^{(m)}(s-t),\quad 
\tilde V_k^{(m)}(x) := \sum_{p \in [\chi ]} \tilde W_k^{(m)}(x,p).
\end{align*}
Then, since it is constant on cubes, $\tilde V_k^{(m)}$ resembles the potential considered for Brownian motion in random scenery, see \cite{Asselah03}.

Lemma \ref{lem:apois_leq_aconst} together with Theorem \ref{theo:lya_quenched} proves Theorem \ref{prop:alphapois_leq_alphaconst}. The well known fact that in the case of constant potential $\eta >0$, $\nu = 0$ $\alpha_\eta(y) = \sqrt{2 \eta} \|y\|_2$ for $y \in \R^d$ shows the equality in Theorem \ref{prop:alphapois_leq_alphaconst}. For convenience we provide a proof for this formula:

We recall the Green function for Brownian motion in constant potential $\eta>0$ ($V = 0$):
\begin{align*}
& g_\eta(x,y)=\frac{2(2\eta)^{\frac{d-2}{2}}}{\sigma_d (d-2)!}\int_1^\infty e^{-\sqrt{2\eta} \|y-x\|_2 t}(t^2-1)^{\frac{d-3}{2}} dt \quad (d \geq 2),\\
& g_\eta(x,y)=\frac{e^{- \sqrt{2 \eta}|y-x|}}{\sqrt{2\eta}} \quad (d=1),
\end{align*}
where $\sigma_d$ is the surface area of the unit ball in $\R^d$ (see for example \cite[Proposition 2.14]{Chung95} and \cite[Paragraph 2.8 Proposition 27]{Dautrey90} or \cite[(5.118) et seqq.]{Stakgold68}). The asymptotic behavior of $g_\eta(x,y)$ is given by
\begin{align}\label{eq:asyptoticlypconst}
\|y-x\|_2^{-\frac{d-1}{2}} e^{-\sqrt{2\eta} \|y-x\|_2} g_\eta(x,y)^{-1} \to C \quad \text{for} \ \|x-y\|_2 \ \to \infty,
\end{align}
where $C >0$. In fact, in the case $d = 1$ this is obvious, for $d \geq 2$ set $l := \|y-x\|_2$, $k:= \sqrt{2\eta}$ and calculate for $l \geq 1$,
\begin{align*}
\int_1^\infty e^{-klt}(t^2-1)^{\frac{d-3}{2}}dt
&= \int_0^\infty e^{-k(v+l)}(\frac {v^2}{l^2} +2 \frac v l)^{\frac{d-3}{2}} \frac {dv} l\\
&= e^{-kl} l^{- \frac{d-1} 2 } \int_0^\infty e^{-kv}(v(\frac v l +2))^{\frac{d-3}{2}} dv,
\end{align*}
where we used the transformation $v=l(t-1)$. Denote by $D_l$ the latter integral. Then $(D_l)_l$ is monotone decreasing in the case $d>3$, is constant in the case $d=3$ and monotone increasing in the case $d=2$. If $d > 3$, $D_l$ can be estimated from above by
$c_1:=\int_0^\infty e^{-kv}(v +2)^{d-3}dv< \infty$ and be bounded from below by
$c_2:= \int_0^\infty e^{-kv}v^{\frac{d-3} 2}dv>0$. For $d=2$, $c_2$ is an upper bound and $c_1$ a lower bound on $D_l$.

Hence  by \eref{eq:asyptoticlypconst} the Lyapunov exponent for constant potential $\eta > 0$ is $\alpha_\eta(y) = \sqrt{2 \eta} \|y\|_2$.

\section*{Proof of the Lower Bound}

Jensen's inequality applied to $a(ny,\eta+V)$ in Theorem \ref{theo:lya_quenched} and Theorem \ref{theo:lya_annealed} shows that $\alpha_{\eta + V} \geq \beta_{\eta + V}$. Hence it suffices to establish the following lower bound for $\liminf_{n \to \infty} \sqrt n \beta_{\eta_n+V_n}$. The proof of the following Proposition parallels the proof of the lower bound in \cite{Kosygina10}.

\begin{proposition}\label{theo:beta_V_geq}
Assume 
$w_\infty :=  \sup_{n \in \N} \|n W_n\|_\infty$, $\xi:= \sup_{n \in \N} \diam \supp ( W_n )$ to be finite. Let $D\geq 0$ such that
\begin{align}\label{eq:assumption1}
\liminf_{n \to \infty} n(\eta_n +  \nu_n \|W_n\|_1 ) \geq D.
\end{align}
Then for $C \geq 0$,
\begin{align*}
\liminf_{n \to \infty} \inf_{y \in \R^d,\, \|y\|_2=C} \sqrt{n} \beta_{\eta_n + V_n}(y) \geq C \sqrt{ 2 D }.
\end{align*}
\end{proposition}

\begin{proof}
Let $n \in \N$. For $C=0$ the statement is trivial, so assume $C >0$. The Lyapunov exponent is a norm on $\R^d$, hence
$
\liminf_{n \to \infty} \inf_{y \in \R^d,\, \|y\|_2=C} \sqrt{n} \beta_{\eta_n + V_n}(y)
$
equals
$
C \liminf_{n \to \infty} \inf_{y \in \R^d,\, \|y\|_2=C} \sqrt{n} \beta_{\eta_n + V_n}(y/\|y\|_2) 
$
and we can restrict ourselves to the case $C = 1$. Furthermore, the underlying Poisson point process is translation invariant, thus a shift of $\chi$ does not modify the Lyapunov exponent. Therefore, by appropriate translations of the functions $W_n$ we suppose without restriction $\suppW := \overline{B(0,\xi)} \supset \bigcup_n \supp W_n$.

Let $r \ge 1$, $n \in \N$ and $y \in \R^d$ with $\|y\|_2=1$. We calculate
\begin{align}\label{eq:teila}
\begin{split}
&\E e(ry, \eta_n +  V_{n}) 
= E[\exp\{- \eta_n H(ry)\} \E [\exp\{ -\int_0^{H(ry)} V_{n}(Z_s)ds\}]]\\
&\qquad =E[\exp\{- \eta_n H(ry)\}\E[\exp\{-\int_{\R^d}
\int_0^{H(ry)}W_{n}(Z_s-x)ds\,\chi(dx)\}]]\\
&\qquad =E[\exp\{- \eta_n H(ry) - \nu_n \int_{\R^d} (1 - \exp\{-\int_0^{H(ry)}W_{n}(Z_s-x)ds\})dx\}],
\end{split}
\end{align}
where we computed the Laplace transform for the random measure $\chi$ (see for instance \cite[Lemma 10.2]{Kallenberg97}).

Define $\Lambda:\R \to \R$ by $\Lambda(t):= 1- \exp\{- t\}$. Fix $\delta \in (0,1)$ and choose $t_0= t_0(\delta) > 0$ such that for $0 \leq t \leq t_0$,
\begin{align} \label{Lambdageqlinear}
\Lambda(t) \geq \delta t.
\end{align}
Let
$
n \geq n_0 := \left\lceil (t_0 /(2 w_\infty))^{-8} \right\rceil
$
and $r \geq 1$.
For $i \in \N_0$ introduce the stopping times
$T_i:= \inf\{s \geq 0 : Z_s \cdot y = in^{1/2} \}$. For $x \in \R^d$ set 
$
i_x:=\min\{i \in \N_0 :  x \cdot y \leq i n^{1/2} + \xi\},
$
for $a \leq b \in \R_{\geq 0}$ abbreviate
$
J_{a}^{b}(x):=  \int_{a}^{b}W_{n}(Z_s-x)ds
$, and define
$
m_r:= \left\lfloor ( r - 1)/n^{1/2} \right\rfloor.
$
We rewrite for $x \in \R^d$ the integral $\int_0^{T_{m_r}}W_{n}(Z_s-x)ds$ as the following sum and truncate:
\begin{align}\label{eq:truncation}
\int_0^{T_{m_r}}W_{n}(Z_s-x)ds 
= \sum_{i= 1}^{m_r} J_{T_{i-1}}^{T_i}(x)
\geq \sum_{i= 1 \vee i_x}^{m_r \wedge \lfloor i_x + n^{1/8}\rfloor} (J_{T_{i-1}}^{T_i}(x) \wedge (w_\infty n^{-1/4})).
\end{align}
Then the term on the right hand side of \eref{eq:truncation} is less than or equal to 
\begin{align}\label{eq:Lambda_applicable}
(n^{1/8}+1) w_\infty n^{-1/4}
\leq 2 w_\infty n^{-1/8} \leq  t_0.
\end{align}
Observe that
\begin{align}\label{eq:H_geq_T}
H(ry) \geq T_{m_r}.
\end{align}
Indeed $\overline{ B(ry, 1)} \subset \HH:= \{x \in \R^d: xy \geq m_r n^{1/2}\}$, hence $\inf\{ s \geq 0: Z_s \in \overline{ B(ry, 1)}\} \geq \inf\{ s \geq 0: Z_s \in \HH\}$. Monotonicity of $\Lambda$ and inequalities \eref{Lambdageqlinear}, \eref{eq:truncation} and \eref{eq:Lambda_applicable} yield
\begin{align}\label{eq:Lambda_leq}
\Lambda(J_{0}^{H(ry)}(x))
\geq  \delta \sum_{i= 1 \vee i_x}^{m_r \wedge \lfloor i_x + n^{1/8}\rfloor} (J_{T_{i-1}}^{T_i}(x) \wedge (w_\infty n^{-1/4})).
\end{align}
Note that $\{x \in \R^d: i_x \leq i \text{ and } i \leq i_x+n^{1/8}\}$\\
= $\{x \in \R^d: x \cdot y \leq in^{1/2}+\xi  \text{ and } x\cdot y>(i-n^{1/8}-1)n^{1/2}+\xi \}$.
Hence, by \eref{eq:Lambda_leq},
\begin{align}\label{eq:teilb}
 &\int_{\R^d} (1- \exp\{-\int_0^{H(ry)}W_{n}(Z_s-x)ds\})dx
\geq \sum_{i= 1 }^{m_r } Y_{i},
\end{align}
where
\begin{align*}
Y_{i}
&:= \delta \int_{\SSs_i} J_{T_{i-1}}^{T_i}(x) \wedge (w_\infty n^{-1/4})dx,\\
\SSs_i&:=\{x \in \R^d:-n^{5/8}+\xi  < x \cdot y- (i-1)n^{1/2} \leq n^{1/2}+ \xi\}.
\end{align*}
The sequence
$
(\eta_n(T_i - T_{i-1}) + \nu_n Y_{i})_i
$
is i.i.d. under $P$. In fact
\begin{align*}
\delta^{-1} Y_{i}
&= \int_{\SSs_i} \int_{T_{i-1}}^{T_{i}}W_{n}(Z_s-x)ds \wedge (w_\infty n^{-1/4}) dx\\
&= \int_{\{z: -n^{-5/8} + \xi < z \cdot y \leq n^{1/2} + \xi\}} \int_{T_{i-1}}^{T_{i}}W_{n}(Z_s-Z_{T_{i-1}}-z)ds \wedge (w_\infty n^{-1/4}) dz\\
&= \int_{\{z: -n^{-5/8} + \xi < z \cdot y \leq n^{1/2} + \xi\}} \int_{0}^{\tau_{i}}W_{n}(Z_{s+T_{i-1}}-Z_{T_{i-1}}-z)ds \wedge (w_\infty n^{-1/4}) dz,
\end{align*}
where we used the transformation $z= x - Z_{T_{i-1}}$, and $\tau_i:= \inf\{s \geq 0 : (Z_{s+ T_{i-1}}-Z_{T_{i-1}}) \cdot y = n^{1/2}\}$. Further, $T_i-T_{i-1} = \tau_i$, thus $\eta_n(T_i - T_{i-1}) + \nu_n Y_{i}$ is a function of the process $(Z_{s + T_{i-1}} - Z_{T_{i-1}})_{s \geq 0}$. An application of the strong Markov property, i.e.\ of the fact that $(Z_{s + T_{i-1}} - Z_{T_{i-1}})_{s \geq 0 }$ is independent of $\FF_{T_{i-1}}$ and distributed like $(Z_s)_{s\geq 0}$ (see for instance \cite[Theorem 11.11]{Kallenberg97}) proves the statement. Thus, \eref{eq:teila}, \eref{eq:H_geq_T} and \eref{eq:teilb} imply
\begin{align}\label{eq:e_leq}
\E e(ry, \eta_n +  V_{n,k}) 
\leq E[\exp\{- \sum_{i= 1 }^{m_r } ( \eta_n \tau_i  - \nu_n Y_{i}) \}]
= E[\exp\{- \eta_n T_1 - \nu_n Y_{1}\}]^{m_r}.
\end{align}
\begin{figure}[!tb]
 \begin{center}
  \includegraphics{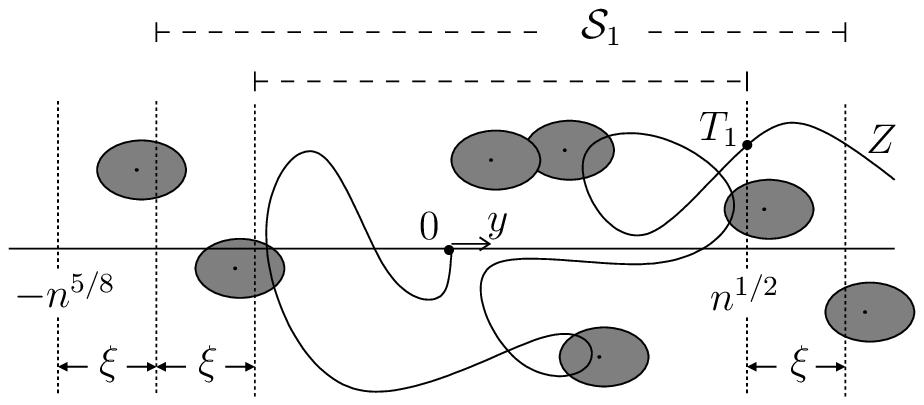}
 \end{center}
\caption{On the event $\AAa_n$ $Z$ is restricted to stay in the strip $\{x \in \R^d: -n^{5/8}+2\xi < x \cdot y \leq n^{1/2}\}$ until $T_1$. In this situation $Z$ can only feel obstacles located in $\SSs_1$.}
\end{figure}
For $\QQ \subset \R^d$ measurable and $T$ a stopping time we denote by $L_\QQ^T$ the time spent by $Z$ in $\QQ$ until time $T$, i.e.\ $L_\QQ^T = \int_0^T 1_\QQ(Z_s)ds$. On the event
\begin{align*}
\AAa_n:=\{Z_s\cdot y > - n^{5/8} +2 \xi \text{ for all } s \leq T_1\}
\cap\{ L_{\suppW +x}^{T_1} \leq n^{3/4} \text{ for all } x \in \R^d\}
\end{align*}
we have 
\begin{align}\label{eq:Y_1=}
Y_{1} = \delta T_1 \|W_{n}\|_1.
\end{align}
Indeed, on $\AAa_n$ for $s \leq T_1$ we have $Z_s - \supp W_n \subset \SSs_1$. Since $x \in Z_s - \supp W_n$ if and only if $W_n(Z_s-x) \neq 0$, this shows that on $\AAa_n$ for $s \leq T_1$ we have $\int_{\SSs_1} W_{n}(Z_s-x)dx = \|W_n\|_1$. Hence, by Fubini's Theorem, on $\AAa_n$ we get $ \int_{\SSs_1} \int_0^{T_1} W_{n}(Z_s-x)ds\,dx = T_1 \|W_{n}\|_1$. Since $\supp W_n \subset \suppW$ and $n \|W_n\|_\infty < w_\infty$, on $\AAa_n$
\[
\int_0^{T_{1}}W_{n}(Z_s-x)ds \leq n^{-1} w_\infty L_{\suppW+x}^{T_1} \leq w_\infty n^{-1/4}.
\]
This shows $Y_{1} = \delta \int_{\SSs_1} \int_0^{T_1} W_{n}(Z_s-x)ds\,dx$ on $\AAa_n$, and \eref{eq:Y_1=} follows.
Using the Laplace transform for one-dimensional Brownian motion hitting times (see for instance \cite[(7.4.4)]{Durrett05}), we get by \eref{eq:e_leq} and \eref{eq:Y_1=} for $n \geq n_0$, $r \ge 1$, 
\begin{align*}
\E e(ry, \eta_n + V_{n})
& \leq (E[\exp\{- \eta_n T_1-  \delta \nu_n \|W_{n}\|_1 T_1 \}, \AAa_n]+P[\AAa_n^c])^{m_r}\\
& \leq (\exp\{- \sqrt{2 n \eta_n + 2 n \delta \nu_n \|W_{n}\|_1 } \}+P[\AAa_n^c])^{m_r},
\end{align*}
which implies
\begin{align*}
-\frac{1}{r}\ln \E e(ry,\eta_n+V_n)
& \geq - \frac{m_r}{r}\ln (\exp\{- \sqrt{ 2 n\eta_n + 2 n \delta \nu_n \|W_{n}\|_1}\}+P[\AAa_n^c]).
\end{align*}
Therefore, taking the limit in $r$ Theorem \ref{theo:lya_annealed} shows for $n \geq n_0$,
\begin{align}\label{eq:beta_ge}
\sqrt n \beta_{\eta_n+V_n}(y)
& \geq - \ln (\exp\{- \sqrt{2 n \eta_n+ 2 n \delta \nu_n  \|W_n\|_1}\}+P[\AAa_n^c]).
\end{align}
By rotational invariance of the Brownian law, $P[\AAa_n^c]$ does not depend on the direction of $y$. In addition, the choice of $n_0$ was independent of $y$. Therefore, for $n \geq n_0$, for all $y\in \R^d$ with $\|y\|_2=1$,
$
\inf_{y \in \R^d, \, \|y\|_2=1} \sqrt n \beta_{\eta_n+V_n}(y)
$
is greater or equal the right hand side of \eref{eq:beta_ge}.
We take the limes inferior in $n$ and use the fact that $P[\AAa_n^c] \to 0$ for $n \to \infty$ (see \eref{eq:PA_n_to_0}). By assumption \eref{eq:assumption1} we obtain
\begin{align*}
\liminf_n \inf_{y \in \R^d, \, \|y\|_2=1} \sqrt n \beta_{\eta_n + V_n}(y)
&\geq \liminf_n \sqrt{ 2 n\eta_n + 2 n \delta \nu_n  \|W_n\|_1}\\
&\geq \delta \liminf_n \sqrt{ 2 n\eta_n + 2 n \nu_n  \|W_n\|_1}
\geq \delta \sqrt{2 D}.
\end{align*}
Taking the supremum over $\delta \in (0,1)$ shows the statement.

To complete the proof it remains to show
\begin{align}\label{eq:PA_n_to_0}
P[\AAa_n^c] \to 0 \ \text{for} \ n \to \infty.
\end{align}
First
\begin{align*}
P[\AAa_n^c] 
\leq P[\exists s \leq T_1: Z_s\cdot y \leq - n^{5/8} + 2\xi ] + 
P[\exists x \in \R^d: L_{\suppW+x}^{T_1} > n^{3/4} ].
\end{align*}
The first term on the right-hand side is the probability that a one-dimensional Brownian motion hits
$-n^{5/8} +2\xi$ before hitting $n^{1/2} $. This probability is smaller than
$n^{1/2}/(n^{1/2}+n^{5/8}-2\xi)$ (see for example \cite[Theorem 7.5.3]{Durrett05}) and therefore converges to zero as $n$ goes to infinity.

In order to examine the second term, we will reduce the problem to a one-dimensional setting. For this purpose, notice that
$
Z_s \in \suppW$ implies $\pi_y(Z_s) \in \pi_y(\suppW)
$
where $\pi_y$ is the projection onto $y\R$. Hence
\begin{align*}
L_{\suppW+x}^{T_1} = \int_0^{T_1}1_{\suppW+x}(Z_s)ds 
\leq \int_0^{T_1}1_{\pi_y(\suppW+x)}(\pi_y(Z_s))ds.
\end{align*}
Since the term on the right side may be interpreted as the time spent by a one-dimensional Brownian motion in a ball with radius $\xi$ and center $\pi_y(x)$ before hitting $n^{1/2}$, we assume without restriction in the following $Z=(Z_s)_s$ to be a one-dimensional Brownian motion, the set $\suppW$ to be the interval $[-\xi, \xi]$ and $T_1$ to be the hitting time of $n^{1/2}$. We then have to show convergence to zero of $P[\exists x \in \R: L_{\suppW + x}^{T_1} > n^{3/4}]$ for $n \to \infty$.
Note that for dimension $d =1$ the occupation times formula shows $P$-a.s.\ $L_{\suppW+x}^{T_1} = \int_{\suppW+x} l_z^{T_1} dz$, where $(l_z^t)_{t,z}$ is the local time process of one-dimensional Brownian motion (see for instance \cite[Corollary 6.1.6]{Revuz91} and the first remark thereafter).

Using scaling relations for the local time process (see \cite[Exercise 6.2.11]{Revuz91}) we get
\begin{align*}
L_{\suppW + x}^{T_1} = \int_{\suppW + x} l_z^{T_1} dz \overset{d}{=} n^{1/2} \int_{\suppW + x} l_{n^{-1/2}z}^{H(\{1\})} dz
\end{align*}
where $H(\{1\})$ denotes the hitting time of the set $\{1\}$. The local time process is $P$-a.s.\ bounded, i.e.\ $\sup_{z \in \R} l_z^{H(\{1\})} < \infty$ $P$-a.s.. Indeed, $P$-a.s.\ $z \mapsto l_z^{H(\{1\})}$ is a density for the occupation measure $\mu :  A \mapsto \int_0^{H(\{1\})} 1_A(Z_s)ds$ (use \cite[Corollary 6.1.6]{Revuz91}). Since $P$-a.s.\ $H(\{1\})$ is finite, the range of $Z_s$ up to time $H(\{1\})$ is bounded, thus $\mu$ has $P$-a.s.\ compact support $\supp \mu$. Furthermore, $P$-a.s.\ $(z,t) \mapsto l_z^t$ is jointly continuous  (see \cite[Theorem 6.1.7]{Revuz91}), hence $z \mapsto l_z^{H(\{1\})}$ is a continuous function with compact support and therefore bounded. We get
\begin{align*}
P[\exists x \in \R: L_{\suppW + x}^{T_1} > n^{3/4}]
&= P[\exists x \in \R: \int_{\suppW + x} l_{n^{-1/2}z}^{H(\{1\})} dz > n^{1/4}]\\
&\leq P[\leb(\suppW) \| l_{\cdot}^{H(\{1\})} \|_\infty > n^{1/4}].
\end{align*}
Finiteness of $\| l_{\cdot}^{H(\{1\})} \|_\infty$ and $\sigma$-continuity of $P$ yield convergence to $0$ of the latter.
\end{proof}

\vspace{0.5 cm}
\noindent
\textbf{Acknowledgements:}
This work was funded by the ERC Starting Grant 208417-NCIRW.\\
I am glad to express my thanks to my supervisor Prof. Martin P. W. Zerner for proposing me this interesting problem, for many helpful advises and for his steady support. I thank Dr. Elmar Teufl for many inspiring discussions concerning this subject, and especially for the help with \eref{eq:asyptoticlypconst}.

\vspace{0.5 cm}
\noindent Johannes Rue\ss\\
Mathematisches Institut, Universit\"at T\"ubingen\\
Auf der Morgenstelle 10, 72076 T\"ubingen, Germany\\
Email: johannes.ruess@uni-tuebingen.de

\small
\bibliographystyle{alpha}

\end{document}